\documentclass[12pt,a4paper]{article}
\usepackage{latexsym,amsthm,amssymb,amscd,amsmath}
\newcounter{alphthm}
\theoremstyle{plain}
\newtheorem{thm}{Theorem}[section]

 \newtheorem{exam}[thm]{Example}
 
 \theoremstyle{definition}
 \newtheorem{defn}[thm]{Definition}
 \theoremstyle{remark}
 \newtheorem{rem}[thm]{\bf Remark}
 \numberwithin{equation}{subsection}
\begin{document}

\title {Kannan Fixed-Point Theorem On Complete Metric Spaces And On Generalized Metric Spaces Depended an
Another Function  }

\author{ S. Moradi \\\\
Faculty  of Science, Department of Mathematics\\
Arak University, Arak,  Iran\\
\date{}
 }
 \maketitle
\begin{abstract}
We obtain sufficient conditions for existence of unique fixed
point of Kannan type mappings on complete metric spaces and on
generalized complete metric spaces depended an another function.
\end{abstract}


\section{Introduction}

The fixed point theorem most be frequently cited in Banach
condition mapping principle (see [4] or [6]), which asserts that
if $(X,d)$ is a complete metric space and $S:X \longrightarrow X$
is a contractive mapping ($S$ is contractive if there exists $k
\in [0,1)$ such that for all $x,y \in X$, $d(Sx,Sy) \leq
kd(x,y)$) then $S$ has a unique fixed point.\\
In 1968 [5] Kannan established a fixed point theorem for mapping
satisfying:
\[ d(Sx,Sy) \leq \lambda\big{[}d(x,Sx)+d(y,Sy)\big{]}, \hspace{7cm} (1)\]
for all $x,y \in X$, where $\lambda \in [0,\frac{1}{2})$.\\
Kannan's paper [5] was followed by a spate of papers containing a
variety of contractive definitions in metric spaces.\\
Rhoades [7] in 1977 considered 250 type of contractive
definitions and analyzed the relationship among them.\\
In 2000 Branciari [3] introduced a class of generalized metric
spaces by replacing triangular inequality by similar ones which
involve four or more points instead of three and improved Banach
contraction mapping principle.\\
Recently Azam and Arshad [1] in 2008 extended the Kannan's
theorem for this kind of generalized metric spaces.\\
In 2009 [2] A. Beiranvand, S. Moradi, M. Omid and H. Pazandeh
introduced new classes of contractive functions and established
the Banach contractive principle.\\
In the present paper at first we extend the Kannan's theorem [5]
and then extend the theorem due to Azam and Arshad [1] for these
new classes of functions.

From the main results we need some new definitions.

\begin{defn}$[2]$
 Let $(X,d)$ be a metric space. A mapping $T:X \longrightarrow X$
 is said sequentially convergent if we have, for every sequence
 $\{y_{n}\}$, if $\{Ty_{n}\}$ is convergence then $\{y_{n}\}$ also is
 convergence. $T$ is said subsequentially convergent if we have, for
 every sequence $\{y_{n}\}$, if $\{Ty_{n}\}$ is convergence then
 $\{y_{n}\}$ has a convergent subsequence.
\end{defn}

\begin{defn}$([1]$ or $[3])$
Let $X$ be a nonempty set. Suppose that the mapping $d:X
\longrightarrow X$, satisfies:\\
(i) $d(x,y) \geq 0$, for all $x,y \in X$ and $d(x,y)=0$ if and
only if $x=y$;\\
(ii)$d(x,y)=d(y,x)$ for all $x,y \in X$;\\
(iii)$d(x,y) \leq d(x,w)+d(w,z)+d(z,y)$ for all $x,y \in X$ and
for all distinct points $w,z \in X \backslash
\{x,y\}$[rectangular property].\\
Then $d$ is called a generalized metric and $(X,d)$ is a
generalized metric space.
\end{defn}

For more information can see [1] and [3].
%-------------------------------------------------------------------------------------------

\section{Main Results}

In this section at first we extend the Kannan's theorem [5] and
then extend the Azam and Arshad theorem [1].

\begin{thm}$[$Extended Kannan's Theorem$]$
 Let $(X,d)$ be a complete metric space and  $T,S:X \longrightarrow
 X$ be mappings such that $T$ is continuous,  one-to-one and  subsequentially
 convergent. If $\lambda \in [0,\frac{1}{2})$ and
 \[ d(TSx,TSy) \leq \lambda \big{[}d(Tx,TSx)+d(Ty,TSy)\big{]}, \hspace{5cm} (2)\]
 for all $x,y \in X$, then $S$ has a unique fixed point. Also if
 $T$ is sequentially convergent then for every $x_{0} \in X$ the
 sequence of iterates $\{S^{n}x_{0}\}$ converges to this fixed
 point.
\end{thm}
\begin{proof}
Let $x_{0}$ be and arbitrary point in $X$. We define the
iterative sequence $\{x_{n}\}$ by $x_{n+1}=Sx_{n}$ (equivalently,
$x_{n}=S^{n}x_{0}$), $n=1,2,...$. Using the inequality (2), we
have
\begin{eqnarray*}
d(Tx_{n},Tx_{n+1})&=&d(TSx_{n-1},TSx_{n})\\
&\leq& \lambda \big{[}d(Tx_{n-1},TSx_{n-1})+d(Tx_{n},TSx_{n})
\big{]}, \hspace{3cm} (3)
\end{eqnarray*}
so,
\[
d(Tx_{n},Tx_{n+1}) \leq \frac{\lambda}{1-\lambda}
d(Tx_{n-1},Tx_{n}). \hspace{6.7cm} (4)\] By the same argument,
\begin{eqnarray*}
d(Tx_{n},Tx_{n+1}) &\leq& \frac{\lambda}{1-\lambda}
d(Tx_{n-1},Tx_{n}) \leq (\frac{\lambda}{1-\lambda})^{2}
d(Tx_{n-2},Tx_{n-1})\\
 &\leq& ... \leq (\frac{\lambda}{1-\lambda})^{n} d(Tx_{0},Tx_{1}). \hspace{5.2cm}
(5)
\end{eqnarray*}
By (5), for every $m,n \in \Bbb{N}$ such that $m>n$ we have,
\begin{eqnarray*}
d(Tx_{m},Tx_{n}) &\leq&
d(Tx_{m},Tx_{m-1})+d(Tx_{m-1},Tx_{m-2})+...+ d(Tx_{n+1},Tx_{n})\\
&\leq& \big{[}(\frac{\lambda}{1-\lambda})^{m-1}+
(\frac{\lambda}{1-\lambda})^{m-2}+...+
(\frac{\lambda}{1-\lambda})^{n}\big{]}
d(Tx_{0},Tx_{1}) \\
 &\leq& \big{[} (\frac{\lambda}{1-\lambda})^{n}+(\frac{\lambda}{1-\lambda})^{n+1}+...
 \big{]}d(Tx_{0},Tx_{1})\\
 &=& (\frac{\lambda}{1-\lambda})^{n} \frac{1}{1-(\frac{\lambda}{1-\lambda})} d(Tx_{0},Tx_{1}). \hspace{4.5cm}
(6)
\end{eqnarray*}
Letting $m,n \longrightarrow \infty$ in (6), we have $\{Tx_{n}\}$
is a Cauchy sequence, and since $X$ is a complete metric space,
there exists $v \in X$ such that
\[ \underset{n \rightarrow \infty}{\lim}Tx_{n}=v. \hspace{11cm} (7)\]
Since $T$ is a subsequentially convergent, $\{x_{n}\}$ has a
convergent subsequence. So there exists $u\in X$ and $\{x_{n(k)}\}
_{k=1}^\infty$ such that  $\underset {k \rightarrow \infty}{\lim}
x_{n(k)}=u$.\\
Since $T$ is continuous and $\underset {k \rightarrow
\infty}{\lim} x_{n(k)}=u$, $\underset {k \rightarrow
\infty}{\lim} Tx_{n(k)}=Tu$.\\
By (7) we conclude that $Tu=v$. So
\begin{eqnarray*}
d(TSu,Tu) &\leq&
d(TSu,TS^{n(k)}x_{0})+d(TS^{n(k)}x_{0},TS^{n(k)+1}x_{0})+ d(TS^{n(k)+1}x_{0},Tu)\\
&\leq& \lambda
\big{[}d(Tu,TSu)+d(TS^{n(k)-1}x_{0},TS^{n(k)}x_{0})\big{]}\\
&+&(\frac{\lambda}{1-\lambda})^{n(k)}d(TSx_{0},Tx_{0})+d(Tx_{n(k)+1},Tu)\\
&=& \lambda d(Tu,TSu)+ \lambda
d(Tx_{n(k)-1},Tx_{n(k)})\\
&+&(\frac{\lambda}{1-\lambda})^{n(k)}d(Tx_{1},Tx_{0})+d(Tx_{n(k)+1},Tu),
\hspace{3cm} (8)
\end{eqnarray*}
hence,
\begin{eqnarray*}
d(TSu,Tu) &\leq&
\frac{\lambda}{1-\lambda}d(Tx_{n(k)-1},Tx_{n(k)})+\frac{1}{1-\lambda}(\frac{\lambda}{1-\lambda})^{n(k)}d(Tx_{1},Tx_{0})\\
&+&\frac{1}{1-\lambda}d(Tx_{n(k)+1},Tu) \underset{k \rightarrow
\infty }\longrightarrow 0. \hspace{5cm} (9)
\end{eqnarray*}
Therefore $d(TSu,Tu)=0$.\\
Since $T$ is one-to-one $Su=u$. So $S$ has a fixed point.\\
Since (2) holds and $T$ is one-to-one, $S$ has a unique fixed
point.\\
Now if $T$ is sequentially convergent, by replacing $\{n\}$ with
$\{n(k)\}$ we conclude that $\underset{n \rightarrow
\infty}{\lim}x_{n}=u$ and this shows that $\{x_{n}\}$ converges
to the fixed point of $S$.
\end{proof}

\begin{rem}
By taking $Tx \equiv x$ in Theorem 2.1, we can conclude the
Kannan's theorem[5].
\end{rem}
The following example shows that Theorem 2.1 is indeed a proper
extension on Kannan's theorem.
\begin{exam}
Let $X=\{0\}\cup \{\frac{1}{4},\frac{1}{5},\frac{1}{6},...\}$
endowed with the Euclidean metric. Define $S:X \longrightarrow X$
by $S(0)=0$ and $S(\frac{1}{n})=\frac{1}{n+1}$ for all $n \geq
4$. Obviously the condition (1) is not true for every $\lambda >
0$. So we can not use the Kannan's theorem [5]. By define $T:X
\longrightarrow X$ by $T(0)=0$ and
$T(\frac{1}{n})=\frac{1}{n^{n}}$ for all $n \geq 4$ we have, for
$m,n \in \Bbb{N}$ ($m > n$),
\begin{eqnarray*}
|TS(\frac{1}{m})-TS(\frac{1}{n})|&=&\frac{1}{(n+1)^{n+1}}-\frac{1}{(m+1)^{m+1}}\\
&<& \frac{1}{(n+1)^{n+1}} \leq \frac{1}{3}
\big{[}\frac{1}{n^{n}}-\frac{1}{(n+1)^{n+1}}\big{]}\\
&\leq& \frac{1}{3}
\big{[}\frac{1}{n^{n}}-\frac{1}{(n+1)^{n+1}}+\frac{1}{m^{m}}-\frac{1}{(m+1)^{m+1}}\big{]}\\
&=&\frac{1}{3}\big{[}
|T(\frac{1}{n})-TS\frac{1}{n}|+|T(\frac{1}{m})-TS\frac{1}{m}|\big{]}.
\hspace{2.5cm} (10)
\end{eqnarray*}
The inequality (10) shows that (2) is true for $\lambda
=\frac{1}{3}$. Therefore by Theorem 2.1 $S$ has a unique fixed
point.
\end{exam}

In the following theorem we extend the Azam and Arshad theorem
[1].

\begin{thm}
 Let $(X,d)$ be a complete generalizes metric space and  $T,S:X \longrightarrow
 X$ be mappings such that $T$ is continuous,  one-to-one and  subsequentially
 convergent. If $\lambda \in [0,\frac{1}{2})$ and
 \[ d(TSx,TSy) \leq \lambda \big{[}d(Tx,TSx)+d(Ty,TSy)\big{]}, \hspace{5cm} (11)\]
 for all $x,y \in X$, then $S$ has a unique fixed point. Also if
 $T$ is sequentially convergent then for every $x_{0} \in X$ the
 sequence of iterates $\{S^{n}x_{0}\}$ converges to this fixed
 point.
\end{thm}
\begin{proof}

\end{proof}

\begin{rem}
By taking $Tx \equiv x$ in Theorem 2.4, we can conclude the Azam
and Arshad theorem [1].
\end{rem}

The following example shows that Theorem 2.4 is indeed a proper
extension on Azam and Arshad theorem.
\begin{exam}$[1]$
Let $X=\{1,2,3,4\}$. Define $d:X\times X \longrightarrow \Bbb{R}$
as follows:\\
$d(1,2)=d(2,1)=3,$\\
$d(2,3)=d(3,2)=d(1,3)=d(3,1)=1,$\\
$d(1,4)=d(4,1)=d(2,4)=d(4,2)=d(3,4)=d(4,3)=4$.\\
Obviously $(X,d)$ is a generalized metric space and is not a
metric space.\\
Now define a mapping $S:X \longrightarrow X$ as follows:
$$Sx= \left\{\begin{array} {ll}
2 & ; x \neq 1\\
4& ; x=1\\
\end{array}\right. $$
Obviously the inequality (1) is not holds for $S$ for every
$\lambda \in [0,\frac{1}{2})$. So we can not use the Azam and
Arshad theorem for $S$.\\
By define $T:X \longrightarrow X$ by:
$$Tx= \left\{\begin{array} {ll}
2 & ; x=4\\
3& ; x=2\\
4& ; x=1\\
1& ; x=3\\
\end{array}\right. $$
we have
$$TSx= \left\{\begin{array} {ll}
3 & ; x \neq 1\\
2& ; x=1\\
\end{array}\right. $$
We can show that
\[ d(TSx,TSy) \leq \frac{1}{3} \big{[}d(Tx,TSx)+d(Ty,TSy)\big{]}. \hspace{5cm} (12)\]
Therefore by Theorem 2.4, $S$ has a unique fixed point.
\end{exam}

%-----------------------------------------------------------------------------------------------

Email:

S-Moradi@araku.ac.ir

\end{document}